\documentclass[12pt,amstex,leqno]{amsart}

\usepackage{latexsym}
\usepackage{amsthm}
\usepackage{amsmath}

\usepackage{amsfonts}
\usepackage{amssymb}
\usepackage{appendix}
\usepackage[mathcal]{euscript}

\newcommand\op{\operatorname{op}}

\newcommand\Alg{\operatorname{\bf Alg}}
\newcommand\Set{\operatorname{\bf Set}}
\newcommand\Pos{\operatorname{\bf Pos}}

\newcommand{\id}{\operatorname{id}}

\newcommand\colim{\operatorname{\it colim}}

\newcommand{\Mod}{\operatorname{\bf Mod}}

\newcommand{\Sind}{\operatorname{\bf Sind}}

\newcommand\ck{\mathcal {K}}
\newcommand\ca{\mathcal {A}}
\newcommand\cd{\mathcal {D}} 
\newcommand\cp{\mathcal {P}} 
 
\newcommand\cl{\mathcal {L}}
\newcommand\ce{\mathcal {E}}
\newcommand\cv{\mathcal {V}}
\newcommand\cn{\mathcal {N}}
\newcommand\ct{\mathcal {T}}

\newcommand\N{\mathbb{N}}

\input xy
\xyoption{all}

\swapnumbers

\newtheorem{theorem}{Theorem}[section]
\newtheorem{lemma}[theorem]{Lemma}
\newtheorem{birk}[theorem]{Birkhoff Variety Theorem}

\newtheorem{prop}[theorem]{Proposition}
\newtheorem{corollary}[theorem]{Corollary}

\theoremstyle{definition}
\newtheorem{defi}[theorem]{Definition}

\newtheorem{assump}[theorem]{Assumption}
\newtheorem{example}[theorem]{Example}

\newtheorem{remark}[theorem]{Remark}

\newtheorem{nota}[theorem]{Notation}

\numberwithin{equation}{section}

\begin{document}

\title[Varieties of ordered algebras as categories]{Varieties of ordered algebras\\ as categories}
\author{Ji{\v{r}}{\'{\i}} Ad{\'{a}}mek}
\author{Ji{\v{r}}{\'{\i}} Rosick{\'{y}}}
\thanks{Both authors were supported by the Grant Agency of the Czech Republic under the grants 22-02964S and 19-00902S} 

\begin{abstract}
A variety is a category of ordered (finitary) algebras presented by inequations between terms. We characterize categories  enriched over the category of posets which are equivalent to a variety. This is quite analogous to Lawvere's classical characterization of varieties of ordinary algebras.
We also  study the relationship of varieties to discrete Lawvere theories, and varieties as concrete categories over $\Pos$.
 
\end{abstract}

\maketitle

\section{Introduction}
Classical varieties of (finitary) algebras were characterized in the pioneering dissertation of Lawvere \cite{L} as precisely the categories of models of Lawvere theories, see Section~5. Lawvere further characterized varieties  as the categories   with  effective equivalence relations which have an abstractly finite, regularly projective regular generator. In \cite{A} we have simplified this:  a category is equivalent to a variety iff it has

(1) reflexive coequalizers and kernel pairs, and

(2)  an abstractly finite, effectively projective strong generator $G$.

\noindent
Effective projectivity means that the hom-functor of $G$ preserves reflexive coequalizers  -- we recall this in Section~2.
Abstract finiteness is a condition much weaker than finite generation (see Example \ref{E:abstract}).
 We have also  presented another characterization: varieties are precisely the free completions of duals of Lawvere theories under sifted colimits \cite{AR}.

The aim of our paper is to present a categorical characterization of varieties of ordered algebras. These are classes of ordered $\Sigma$-algebras (for finitary signatures  $\Sigma$) that are specified by inequations between terms. Example: ordered monoids, or ordered monoids whose neutral element is the least one.
 Our characterization of categories  equivalent to varieties of ordered algebra turns out to be analogous to the conditions (1) and (2) above,  but we have to work in the realm of \textit{ categories enriched over} $\Pos$, the (cartesian closed) category of posets and monotone maps. That is, a category is  always equipped with a partial order on every hom-set, and composition is monotone. The concepts used by Lawvere need to be modified accordingly. Whereas in ordinary category one works with regular 
epimorphisms (coequalizers of parallel pairs) we work with \textit{subregular epimorphisms} $e\colon X\to Y$ which are the coinserters of parallel pairs (i.e. there exist $u_0$, $u_1\colon U\to X$ such that $e$ is universal w.r.t. $e u_0\leq e u_1$). This leads to a modification  of 
effective  projectivity: we call $G$ subeffective projective if its hom-functor into $\Pos$ preserves reflexive coinserters. Our
 main result characterizes varieties of ordered algebras  via 
strong generators  that  are abstractly finite subeffective projectives.

Power introduced discrete Lawvere theories \cite{P} that, for categories enriched over $\Pos$, we recall in Section~5. We prove that varieties of ordered algebras are precisely the enriched models of discrete Lawvere theories. From this Kurz and Velebil \cite{KV}  derived that they are precisely the free completions of duals of discrete Lawvere theories under (enriched) sifted colimits; we present a simplified proof in Section~5.

We can also view a variety $\cv$ of ordered algebras as a concrete category by considering its  forgetful functor $U\colon \cv \to \Pos$.  Using the above results, we  derive a  characterization of varieties as concrete categories  in Section~6.

\vskip 2mm
\textbf{Related Work.}  Varieties of (possibly infinitary) ordered algebras were studied already in the 1970's by Bloom \cite{B1}, and they were characterized as concrete categories by Bloom and Wright \cite{BW}. In Section~6 we show that the characterization in the latter paper, when restricted  to finitary signatures, is closely related to our main result. However, we have not found an easy way of deriving one of those results from the other one.

Kurz and Velebil published more recently a fundamental paper  on this topic \cite{KV} in which the main subject is the exactness property for categories enriched over $\Pos$. The definition in \cite{KV} is quite natural, but rather involved, based on the technical concept of congruence (Definition 3.8 in loc.cit.). For exact enriched categories the characterization of finitary varieties in their Theorem 5.9 is essentially the same as in our  main result, Theorem 4.7 below. So the main  message of our paper is that one does not need exactness to characterize varieties of ordered algebras as abstract categories. (The exactness is `condensed' into properties of the given generator.) Kurz and Velebil also proved that finitary varieties of ordered algebras are precisely the monadic categories over $\Pos$ for strongly finitary monads; see \cite{ADV} for a simplified proof.

\vskip 3mm
\textbf{Acknowledgements.} The presentation of our results has been improved by suggestions of the referees, for which the authors are very grateful.

\section{Lawvere's Characterization of varieties}\label{sec1}

We recall shortly two of the major results of Lawvere's  famous dissertation: a characterization of categories equivalent to varieties in the classical sense (classes of $\Sigma$-algebras presented by equations) and algebraic theories as a syntax for varieties. Varieties are considered as  categories  with homomorphisms as morphisms.

Let $\Sigma$ be a finitary signature. Given a variety $\ck$ of $\Sigma$-algebras, every set generates  a free algebra of $\ck$. Denote by $G$ the free algebra on one generator. Then $G$ is
\begin{enumerate}
\item[(a)] \textit{abstractly finite} which
means that all copowers $M.G= \coprod\limits_{M} G$ ($M$ a set) exist and every morphism $f\colon G \to M. G$ factorizes through a finite subcopower  ${M'}. G \hookrightarrow {M}. G$ ($M'\subseteq M$ finite),
 
\item[(b)] a \textit{strong  generator with copowers}, i.e., all copowers ${M} . G$ exist and every object $X$ is an extremal quotient of a copower via the canonical map $c_X=[f] \colon \coprod\limits_{f\colon G\to X} G \to X$, and 

\item[(c)] a \textit{regular projective}, i.e., $\ck(G, -)$ preserves regular epimorphisms.
\end{enumerate}

\begin{theorem}[Lawvere \cite{L}, Thm. 2.1]\label{T:Law}
A category is equivalent to a variety of finitary algebras iff it has

\begin{enumerate}
\item[(1)] coequalizers and finite limits,
\item[(2)] effective equivalence relations (i.e. every equivalence relation is a kernel pair of its coequalizer), and
\item[(3)] a regular generator which is an  abstractly  finite regular projective.
\end{enumerate}
\end{theorem}

\begin{remark}\label{R:Law}
 `Coequalizers' in Condition (1) are missing in \cite{L}. This seems to be just a typo: at the end of the proof of Theorem 2.1 Lawvere forms a coequalizer $\overset{=}{r}$ of a parallel pair without commenting why it exists.
\end{remark}

There is another important property of $G$ related to \textit{reflexive coequalizers}. A pair $u$, $v\colon X\to Y$ is called reflexive if it consists of split epimorphisms with a joint splitting $d\colon Y\to X$ ($u\cdot d = v\cdot d=\id$).
A reflexive coequalizer is a coequalizer of a reflexive pair. Whereas in varieties coequalizers are not $\Set$-based in general, reflexive coequalizers are, see \cite{AR}. It then follows that $G$ has the following property introduced by Pedicchio and Wood \cite{PW}:

\setcounter{theorem}{2}

\begin{defi}\label{D:eff} 
An object is called an \textit{effective projective\/} if its  hom-functor preserves reflexive coequalizers.
\end{defi}

 Effective equivalence relations can be deleted from the above theorem, provided that in Condition (3)  we replace regular projective by effective projective.
This was observed by Pedicchio and Wood \cite{PW}. In \cite{A} a full proof of the following  modified theorem is presented:

\begin{theorem}\label{th1.6}
A category is equivalent to a variety of finitary 
algebras iff it has 
\begin{enumerate}
\item[(1)] reflexive coequalizers and kernel pairs, and
\item[(2)] a strong generator which is an abstractly finite effective projective.
\end{enumerate}
\end{theorem}

A pioneering result of Lawvere's thesis was a characterization of varieties as categories of models of algebraic theories.
Let us recall this. 

\begin{nota}\label{N:N}
Denote by
$$
\cn
$$
the  full subcategory of $\Set$ on all natural numbers  $n=\{0, \dots , n-1\}$. We have a canonical strict structure  of finite coproducts in $\cn$: given natural natural numbers $n$ and $k$ we equip $n+k$ with the injections $n\to n+k$ given by $i\mapsto i$, and $k\mapsto n+k$ given by $j\mapsto n+j$.
\end{nota}

\begin{defi}\label{D:th} 
 An \textit{algebraic theory} is a small category $\ct$  whose objects are natural numbers and having 
 finite products together  with a functor $I\colon \cn^{\op} \to \ct$ which is identity on objects and strictly preserves finite products.

 The category $\Mod \ct$ of \textit{models} has as objects functors $A\colon \ct \to \Set$ preserving finite products, and as morphisms natural transformations.
  \end{defi}
 
 \begin{theorem}[Lawvere]\label{T:th} 
  Varieties are, up to equivalence, precisely the categories of models of algebraic theories.
 \end{theorem}
 
 Whereas limits in a variety $\cv$ are computed on the  level of $\Set$ (indeed, they are preserved by the forgetful functor $U\colon \cv\to \Set$), colimits in general are not. However, $U$ preserves directed colimits and (as remarked above) reflexive coequalizers. 
 These two types of colimits are generalized as follows:
 
 \begin{defi}[{\cite{AR}}]\label{D:sif} 
  A small category $\cd$ is called \textit{sifted} if  for all  diagrams $D\colon \cd \to \Set$ in $\Set$ colimits commute with finite products.
  
  A \textit{sifted colimit} in a category is a colimit of a diagram whose domain is a sifted category.
  \end{defi}

  \begin{example}\label{R:sif} 
  Both directed colimits and reflexive coequalizers are sifted colimits. And these two types are, in a way, exhaustive. For example, if $\ck$ and $\cl$ are categories with colimits, then a functor $F\colon \ck \to \cl$ preserves sifted  colimits iff it preserves filtered colimits and reflexive coequalizers, see \cite{AR}.
  \end{example}
  
  \begin{nota}
  For every category $\ck$ denote by $\Sind \ck$ the free completion under sifted colimits: given a category $\cl$ with sifted colimits, every functor $F\colon \ck \to \cl$ has an extension $F'\colon \Sind \ck \to \cl$ preserving sifted colimits, unique up to a natural isomorphism.
\end{nota}

\begin{theorem}[{\cite{AR}}]\label{T:sif} 
Varieties are up to    equivalence precisely the categories $\Sind\ct^{\op}$ for algebraic theories $\ct$.
\end{theorem}

\begin{remark}\label{R2.14}
In \cite{AR} an object $A$ of a category $\ck$ is called \textit{ perfectly presentable} if $\ck(A, -)$ preserves sifted colimits.  If $\ck$ is a variety, these are precisely the retracts of free finitely generated algebras of $\ck$. Moreover, a full subcategory $\ct$ of $\ck$ representing all perfectly presentable objects (up to isomorphism) is an algebraic theory with $\Mod \ct^{\op}$ equivalent to $\ck$.
\end{remark} 

\section{Generators}\label{sec2} 

\begin{assump}\label{as2.1}
From now on we work with categories enriched over the (cartesian closed) category $\Pos$. Thus a category $\ck$ is understood to have partially ordered hom-sets and composition is monotone. Also `functor' means automatically an enriched functor, i.e., one monotone on hom-sets. Natural transformations in the enriched sense are just the ordinary ones. 

Every set is considered as the (discretely ordered) poset.
\end{assump}

\begin{remark}\label{re2.2}
  Recall the concept of a \textit{coinserter} of a parallel (ordered) pair of morphisms
$u_0$, $u_1\colon U\to X$: it is a morphism $f\colon X\to Y$ universal w.r.t. $fu_0 \leq fu_1$. That is,  given $f'\colon X\to Y'$ with $f'u_0 \leq f'u_1$, then (a) there exists $g\colon Y\to Y'$ with $f' =g f$ and (b) for every $\bar{g} \colon Y\to Y'$ from $gf \leq \bar g f$ it follows that $g\leq \bar g$.


\end{remark}

\begin{defi}\label{3.3}
A morphism is called a \textit{subregular epimorphism\/} if it is a coinserter of a parallel pair.
\end{defi}

\begin{example}\label{E:subreg}
(1)
 Every subregular epimorphism $f\colon A\to B$ is an epimorphism -- indeed, it has the stronger property that for parallel pairs $u_1, u_2 \colon B\to C$ we have $u_1\leq u_2$ iff $u_1 f\leq u_2 f$.
 
 If a category  has finite copowers, every regular epimorphism $f\colon A\to B$ is subregular: if $f$ is the coequalizer of $u, v\colon C\to A$, then it is the coinserter of $[u,v]$, $[v, u]\colon C+C\to A$.

(2) In $\Pos$
 subregular epimorphisms are precisely the epimorphisms, i.e., the surjective morphisms. See Proposition \ref{P:surj} for a more general statement.
 
 (3) If If $e=qp$ is a subregular epimorphism, then so is $q$. Indeed, let $u_0$, $u_1$ be a parallel pair with coinserter $e$. It is easy to verify that $q$ is a coinserter of $pu_0$ and $pu_1$.
\end{example}

 \begin{defi}\label{D3.4}
 By a \textit{subkernel pair} of a morphism $f\colon X\to Y$ is meant a parallel pair $u_0$, $u_1 \colon U\to X$  universal w.r.t. $fu_0 \leq fu_1$. That is:
 \begin{enumerate}
 \item[(1)] every pair $v_0$, $v_1 \colon V \to X$ with $fv_0 \leq fv_1$ factorizes as $v_i = u_i \cdot k$ for some $k\colon V\to U$, and
 \item[(2)] whenever $\bar k \colon V\to U$ fulfils $u_i k\leq u_i\bar k$ for $i=0,1$, then $k\leq \bar k$.
 \end{enumerate}
 \end{defi}
 
 \begin{example}\label{E:subker}
(1) In $\Pos$ the subkernel pair of $f\colon X\to Y$ is the pair of projections of the subposet $U$ of $X\times X$ on all $(x_0, x_1)$ with $f(x_0)\leq f(x_1)$.

(2) Every subregular epimorphism $f\colon X\to Y$ is the coinserter of its subkernel pair. Indeed, let $f$ be the coinserter of $v_0$, $v_1\colon V\to X$, and let $k$ be the factorization above. If $g\colon X\to Z$ fulfils $gu_0\leq gu_1$,  then $gv_0=gu_0k \leq gu_1k = gv_1$, thus, $g$ factorizes uniquely through $f$.
\end{example}

\begin{defi}\label{D:tensor} 
Let $K$ be an object of an (enriched) category $\ck$. A \textit{tensor} $P\otimes K$ for a poset $P$ is an object of $\ck$ such that for every object $X$ of $\ck$ we have an isomorphism
\begin{equation}\label{e2.1}
\Pos \big(P, \ck (K, X)\big) \simeq \ck (P\otimes K, X)
\end{equation}
in $\Pos$, natural in $X \in \ck$. We say that $K$ \textit{has tensors} if  $P\otimes K$ exists for every poset $P$.
\end{defi}

\begin{example}
(1) In $\Pos$, $P\otimes K$ is just the categorical product $P\times K$.

(2) In a variety $\ck$ of ordered algebras, let $K$ be the free algebra on one generator. Then $P\otimes K$ is the free algebra on $P$. Indeed, for every algebra $X$ in $ \cv$ the poset $\ck (K,X)$ is precisely the underlying poset $|X|$ of $X$, thus the left hand side of \eqref{e2.1} consist of all monotone functions from $P$ to $|X|$. And they naturally correspond to homomorphism from $P\otimes K$ to $X$.
\end{example}
\begin{nota}\label{no2.5} 
The isomorphism \eqref{e2.1} is denoted by
$$
\xymatrix@R=.21pc{
&P \ar[r]^{f\quad } & \ck (K, X)&\\
\ar@{-}[rrr]&&&\\
&P\otimes K \ar[r]_{\hat f} &X&
}
$$
Example: 
given an object $K$, then for each object $X$ the identity of $\ck (K,X)$ yields the \textit{canonical morphism}
$$
c_X =\widehat{\id} \colon \ck(K,X) \otimes K\to X\,.
$$ 
 The following is an `inverse' example: we define
$$
\eta_P\colon P\to \ck(K, P\otimes K) \quad \mbox{by}\quad \widehat{\eta}_P =\id_{P\otimes K}.
$$
\end{nota}

\begin{example}\label{ex2.7}
For a natural number $n$, considered as the discrete poset $\{0, \dots , n-1\}$, we have a copower
$$
n\otimes K=\coprod_{n} K\,.
$$
We denote it by
$$
n.K\,.
$$

Given $n$ morphisms $f_i\in \ck (K,X)$, the corresponding map $f\colon n\to \ck(K,X)$ yields
$$
\hat f = \widehat{[f_i]} \colon {n}. K\to X\,.
$$
\end{example}

\begin{prop}[{\cite{Bo}, Prop. 6.5.5 and 6.5.6}]\label{P} 
Let $G$ be an object with  tensors. Then the hom-functor
$$
\ck (G, -)\colon \ck \to \Pos
$$
has the following  enriched  left adjoint 
$$
-\otimes G\colon \Pos \to \ck\,.
$$
It assigns to every monotone map $p\colon P\to Q$ the morphism $p\otimes G$ corresponding to 
$ P \xrightarrow{\ p\ } Q \xrightarrow{\, \eta_Q\ } \ck (X, Q\otimes X)$:
$$
p\otimes G = \widehat{\eta_Q\cdot p} : P\otimes X\to Q\otimes X\,.
$$
\end{prop}

\begin{remark}\label{R} 
(1) The naturality of $f \mapsto \hat f$ in Notation~\ref{no2.5} means that this map is monotone $\big( f\leq g\colon P\to \ck (G, X)$ implies $\hat f\leq \hat g\big)$ and for every morphism $h\colon X\to X'$ we have a commutative triangle\\
{\vbox{
$
\xymatrix{
& P\otimes G \ar[ld]_{\hat f} \ar[rd]^{\hat f_h} & \\
X \ar[rr]_{h} && X'
}$ \\
\indent  \vskip -1.5cm \qquad \hskip 3.5cm where\quad $f_h \equiv P\xrightarrow{\ f\ } \ck(G,X) \xrightarrow{\ h\cdot (-)\ } \ck(G, X')$
}}
\vskip 1cm
{\vbox{
(2) For every monotone map $p\colon P\to Q$  in $\Pos$ the following implication holds:
\\
\indent\qquad \quad
$
\xymatrix@C=.5pc{
P \ar[rd]_{f} \ar[rr]^{p} \ar@{}[drr]|{\circlearrowright}&& Q \ar[ld]^{g}\\
& \ck(G,Y) &
}
$
\quad\vskip-1cm \hskip 5cm$\Rightarrow$
\quad \vskip-1.4cm\hskip5.6cm
$\xymatrix@C=.5pc{
P\otimes G \ar[dr]_{\hat f} \ar[rr]^{p\otimes G} \ar@{}[drr]|{\circlearrowright}&& Q\otimes G\ar[dl]^{\hat g}\\
& Y &
}
$
\\[4pt]
This follows from (1) since $p\otimes G= \widehat{\eta_Q \cdot p}$ and for $f= \eta_Q\cdot p$ we get $f_{\hat g} = g\cdot p=f$.
}}
\vskip 1mm
(3) The canonical morphism makes for every monotone map $p\colon P\to \ck(G,X)$ the following  triangle commutative
$$
\xymatrix@C=.3pc{
& P\otimes G \ar[ld]_{p\otimes G} \ar[rd]^{\hat p} & \\
\ck(G,X) \otimes G \ar[rr]_{\quad c_X} && X
}$$ 
This follows from (1): for  $f= \eta_{\ck(G,X)}\cdot p$ we get $\hat f= p\otimes G$ by Proposition \ref{P}. Since $\hat c_X =\id$, we have $f_{c_X}=p$.

\end{remark}

\begin{remark}\label{re2.9}
(1) Recall from \cite{K} the concept of a \textit{weighted colimit} in $\ck$. Given a diagram $D\colon \cd\to \ck$ and a weight $W\colon \cd^{\operatorname{op}} \to \Pos$, a weighted colimit is an object $C$ of $\ck$  with an isomorphism
$$
\ck (C, X) \cong [ \cd^{\op}, \ck] \big(W, \ck (D-, X)\big)
$$
natural in $X \in \ck$. The object $C$ is also denoted by $\colim_W D$.

Example: tensor $P\otimes K$ is a colimit of the diagram $D\colon 1\to \ck$ representing $K$ weighted by $W\colon 1\to \Pos$ representing $P$.

\vskip 1mm
(2) Another example: a coinserter  of $u_0$, $u_1\colon U\to X$ is the colimit of the diagram where $\cd$ is given by a  parallel pair $(\delta_0, \delta_1)$ to which $D$ assigns $(u_0, u_1)$,  and   $W$ assigns  the following monotone  maps
\\
\indent
\hskip 5cm
$\boxed
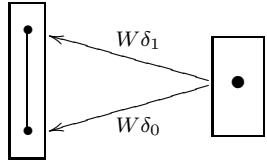
{\xymatrix@R=3.2pc@C=3pc{
{}\ar@{{*}-{*}}[d]\\
{}}}
$
\\
\indent \vskip-2.1cm \hskip 53mm
$\xymatrix@R=1pc@C=5pc{
{}  &\\
{}& \ar[lu]_{W\delta_1} \ar[ld]^{W\delta_0}\\
{}&
}
$
\vskip-1.5cm \hskip7.7cm
$\boxed{\xymatrix@R=0.2pc{
{}\\
{}\bullet\\
{}
}
}
$
 \vskip 5mm
 (3)  Every left adjoint preserves weighted colimits (\cite{Bo}, Prop. 6.7.2).
 
 \vskip 2mm
 (4) The dual concept is a \textit{weighted limit}. Given a diagram $D\colon \cd \to \ck$ and a weight $W\colon \cd \to \Pos$, a weighted limit is an object $L$ with an isomorphism
 $$
\ck(X,L) \simeq [\cd, K] \big(W, \ck(X, D-)\big)
$$
natural in $X$.
 
\vskip 1mm
(5)
Example: the subkernel pair of $f\colon X\to Y$ (Definition~\ref{D3.4}) is the weighted limit where $\cd$ is a cospan $\delta_0$, $\delta_1$ to which
$D$ assigns $D\delta_0 = D\delta_1=f$ and 
 $W\colon \cd \to \Pos$ assigns the embeddings of $\{0\}$ and $\{1\}$ to the chain $0<1$, respectively:
\\
\begin{center}
$
\xymatrix@C=1.5pc@R=0.61pc{
&&&&&& &\{1\}\ar@{_{(}->}[d]&\\
&& {}\ar[dd]^{\delta_1}&& && &\\
&&&{}\ar[rr]^{W}&&&\\
{}\ar[rr]_{\delta_0}&&&&&&\{0\}\hookrightarrow&
}
$
\\
\indent\quad\vskip-1.5cm \hskip 6.3cm
$\boxed{\xymatrix@C=2pc@R=1.52pc{
{}\ar@{{*}-{*}}[d]^>0 ^<1\\
{}}}
$
\end{center}

\vskip 6mm

  \end{remark}


\begin{example}\label{R2}
(1) Every poset $P$ is a \textit{canonical coinserter} of  the projections $\pi_0$, $\pi_1\colon P^{(2)} \to |P|$, where $|P|$ is a discrete poset underlying $P$ and $P^{(2)}$ is the discrete poset of all comparable pairs in $P\times P$. More precisely, the following is a coinserter:
$$
\xymatrix@1{ 
P^{(2)} \ar@<1ex>[r]^{\pi_1}
\ar @<-1ex>[r]_{\pi_0} \ & \ \ |P|\  \ar[r]^{k_p} &\  P
}
$$
where $k_p$ is carried by $\id_{|P|}$.

(2) Consequently, for every object $G$ we have the corresponding coinserter:
$$
\xymatrix@1@C=3pc{ 
P^{(2)} \otimes G\ar@<1ex>[r]^{\pi_1\otimes G}
\ar @<-1ex>[r]_{\pi_0\otimes G} \ & \ \ |P|\otimes G\  \ar[r]^{k_p\otimes G} &\  P\otimes G
}
$$
Indeed, $-\otimes G$ preserves coinserters since they are weighted colimits, see \ref{re2.9}(3).

\vskip 2mm
(3) A \textit{reflexive coinserter} is a coinserter of a reflexive pair $u_0$, $u_1$ (which means that $u_0$, $u_1$ are split epimorphisms with a joint splitting). Observe that the coinserters in (1) and (2) are reflexive: use the diagonal map $|P| \to P^{(2)}$.
\end{example}

\begin{nota}\label{N:str}
(1) Given a full subcategory $\ca$ of $\ck$ (in the enriched sense: the ordering of hom-sets is inherited from $\ck$), we denote by
$$
E\colon \ck \to \Pos^{\ca^{\op}}
$$
 the functor assigning to an object $K$ the restriction  of $\ck (-, K) \colon \ck^{\op} \to \Pos$ to $\ca^{\op}$.
 
 (2) In particular, if $\ca$ consists of a single object $G$, then $\Pos^{\ca^{\op}}$ is the category of posets with a (monotone) action of the ordered monoid $\ck(G,G)^{\op}$. Morphisms are the monotone equivariant maps. Here the functor $E$ assigns to $K$ the poset $\ck(G,K)$ with  the action corresponding to $u\colon G\to G$ given by precomposition with $u$.
 \end{nota}
 
 \begin{defi}\label{D:emb}
 A morphism $f\colon X\to Y$ is  called 
 \begin{enumerate}
 \item[(a)] an \textit{embedding} if given morphisms $u_0, u_1 \colon U\to X$ we have $u_0\leq u_1$ iff $fu_0 \leq fu_1$
 \end{enumerate}
 and
 \begin{enumerate}
 \item[(b)] an \textit{extremal epimorphism} if whenever it  factorizes through an embedding  $m\colon Y_0 \to Y$, then $m$ is invertible.
 \end{enumerate}
 \end{defi}
 
 In ordinary categories an object $G$ is called a strong generator provided that every monomorphism $m\colon X\to Y$ such that 
$$
m.(-) \colon \ck(G,X)\to \ck(G,Y)
$$ 
is a bijection is invertible. In case $G$ has copowers, this is equivalent to each canonical map $c_K\colon \coprod\limits_{f\colon G\to K} G\to K$ being  an extremal epimorphism (one not factorizing through a proper subobject of $K$). Here is the enriched version:
 
 \begin{defi}[Kelly\cite{K}]\label{D:str}
 An object $G$ of $\ck$ is a \textit{strong generator} provided that the functor $E$ of  Notation \ref{N:str}(1) is conservative: a morphism  $m$ of $\ck$ is invertible  iff $Em$ is.
 \end{defi}
 
 \begin{prop}\label{P:str}
 Let $G$ be a generator with tensors. Then it is a strong generator iff all canonical maps
 $$
 c_K \colon \ck(G,K) \otimes G\to K \qquad (K\in \ck)
 $$
 are extremal epimorphisms.
 \end{prop}
 \begin{proof}
 (1) Sufficiency. We are to prove that given a morphism $h\colon K\to L$ with $Eh$ invertible, then $h$ is invertible. By assumption we have a monotone map $p\colon \ck(G,L)\to \ck(G,K)$ inverse to $h.(-)$. Since $G$ is a generator, this implies that $h$ is an embedding: given $u_0$, $u_1$ in $\ck(G, K)$ then $hu_0 \leq hu_1$ iff $u_0 \leq u_1$. Thus all we need to verify is that $c_L$ factorizes through $h$. By Remark \ref{R}(1) the composite of $\hat p\colon \ck(G,L) \otimes G\to K$ and $h$ is $\hat p_h$, where  $p_h\colon \ck(G,L) \to \ck(G,L)$ sends $u\colon G\to L$ to $h\big(p(u)\big)=u$. Thus $p_h=\id_{\ck(G,L)}$. By definition $c_L=\widehat{\id}_{\ck(G,L)}$. We conclude that the triangle below commutes
 $$
\xymatrix{
\ck(G,L) \otimes G\ar[rr]^{c_L}\ar[dr]_{\hat p} && L\\
& K \ar [ur]_{h}&
}
$$
Therefore $h$ is invertible.

(2) Necessity. If $G$ is a strong generator, and if 
$$
c_K = mh \quad \mbox{for an embedding} \quad m\colon M\to K\,,
$$
we are to prove that $m$ is invertible. By assumtion, we just need to prove that $Em =m.(-)\colon \ck(G,M) \to \ck(G,K)$ is invertible. Every morphism $f\colon G\to K$ yields a unique (monotone) map $p\colon 1\to \ck (G,K)$ with $f= \hat p$ (using $1 \otimes G=G$). The following diagram commutes:
$$
\xymatrix{
G \ar[r]^{f}
\ar[d]_{p\otimes G} & K\\
\ck (G,K) \otimes G \ar[ur]^{c_K} \ar[r]_{\hskip 6mm h} & M \ar[u]_{m}
}
$$
due to Remark \ref{R} (3). We thus obtain a mapping
$$
d\colon \ck(G,K) \to \ck (G,M)\,, \quad d(f) = h(p\otimes G)\,.
$$
The above diagram yields $Em\cdot d=\id$. Since $m$ is an embedding in $\ck$, $Em$ is one in $\Pos$. Thus $Em= d^{-1}$.
\end{proof}

For the following  definition we recall the concept of a \textit{slice category} $\ca/K$. Let $\ca$ be a subcategory of $\ck$ and $K\in  \ck$.  The objects of $\ca/K$ are the morphisms  $a\colon A\to K$ with $A\in \ca$ in $\ck$. Morphisms to another  object $a' \colon A'\to K$ are the morphisms $f\colon A \to A'$ of $\ca$ with $a=fa'$, and their ordering is inherited from $\ck$. We have the forgetful functor $D\colon \ca /K \to \ck$ given by $D(A,a) =A$. 

\begin{defi}\label{R:dense}
A full subcategory $\ca$ of $\ck$ is called  \textit{dense} if it satisfies one of the conditions below -- they are equivalent  by \cite{K}, Theorem 5.1:

1. The functor $E\colon \ck\to \Pos^{\ca^{\op}}$ of Notation \ref{N:str} is full and faithful ($f\leq g$ iff $Ef \leq Eg$ for parallel pairs $f$, $g$).

2. (a) For every object $K\in \ck$ and every cocone of $D\colon \ca/K\to \ck$ with
 codomain $L$,  there exists a morphism $k\colon K\to L$ such that the given  cocone is $(k\cdot a)_{(A,a) \in \ca/K}$, and

\vskip 1mm
(b) given a morphism $\bar k\colon K\to L$ with $k\cdot a \leq \bar k \cdot a$ for all $(A,a)$, it follows  that $k\leq \bar k$.
\end{defi}

The concept of an abstractly finite object (Section~2) has the following variant in enriched categories.


\begin{example}\label{E:abstract} 
(1) A poset is abstractly finite in $\Pos$ (see the beginning of Section 2) iff it has only finitely many connected components. Thus there exist abstractly finite posets of an arbitrarily large cardinality.

(2) A free algebra of a variety on an abstractly finite poset is abstractly finite.

(3) Every finitely generated object $G$ in a category $\ck$ is abstractly finite (but not conversely, as we have seen). Indeed, a copower $P.G$ with $P$ infinite is a directed colimit of all subcopowers $Q.G$ with $Q\hookrightarrow P$ finite and nonempty. Given two such objects $Q.G$ and $Q'.G$, the connecting morphism, in case $Q\leq Q'$, is $i.G$ for the inclusion $i\colon Q \hookrightarrow Q'$. Since $i$ is a split monomorphism in $\Set$, $i.G$ splits in $K$. Thus $\ck(G, -)$ preserves the above  directed colimit, proving  that $G$ is abstractly finite.
\end{example}

\begin{defi}\label{Def3.15} 
 An object $G$ is  \textit{subregularly projective\/} if its hom-functor $\ck(G,-) \colon \ck \to \Pos$ preserves subregular epimorphisms. That is, given a subregular epimorphisms $e\colon X\to Y$, every morphism from $G$ to $Y$ factorizes through $e$.
\end{defi}

\begin{lemma}\label{L:ca}
Let $\ck$  be a category with subkernel pairs and reflexive coinserters. For every subregularly projective strong generator $G$ with copowers all canonical maps (Notation 
 \ref{no2.5}) are subregular epimorphisms.
 \end{lemma}
\begin{proof}
Let $u_0$, $u_1$ be a subkernel pair of $c_X$:
$$
\xymatrix@C=2pc@R=3pc{
U \ar@<.7ex>[r]^{u_1\qquad} \ar@<-.7ex>[r]_{u_0 \qquad} 
& \ck(G,X) \otimes G \ar[dr]_e \ar[rr]^{c_X} & &X\\
 & G \ar[ul]^{k} \ar@<-.5ex>[u]^{w_0\ }\ar@<.5ex>[u]_{\ w_1\ }
 \ar@<.7ex>[r]^{v_1\quad } \ar@<-.7ex>[r]_{v_0\quad } & Y \ar[ur]_m &
 }
 $$
  Since $u_0$, $u_1$ is obviously reflexive, we can form its coinserter $e$, and we get a unique $m\colon Y\to X$ with $c_X = me$. Our task is to prove that $m$ is invertible.
Since $c_X$ is an extremal epimorphism (Proposition \ref{P:str}) we just need to 
  prove that $m$ is an embedding. As $G$ is a generator, this amounts to showing that given $v_0$, $v_1\colon G\to Y$ with $mv_0 \leq mv_1$, it follows that $v_0\leq v_1$. Since $G$ is  projective w.r.t. subregular epimorphisms,
  there exist morphisms $w_i\colon G\to \ck(G,X) \otimes G$ with 
$v_i = e\cdot w_i$ ($i=0,1$). From $c_X u_0 \leq c_X u_1$ we conclude $c_X w_0 \leq c_X w_1$ (using that $c_X w_i= m e w_i = mv_i$). Therefore, there exists $k\colon G \to U$ with $w_i = u_i k$. This proves $v_0 \leq v_1$ as desired: $v_i = ew_i = eu_i k$.
\end{proof}
\begin{theorem}\label{P1} 
Let $\ck$ be a category with subkernel pairs and reflexive coinserters.
If $G$ is  an abstractly finite subregularly projective strong 
generator, then the full subcategory of all finite copowers $ n. G$,\
 $n\in \N$, is dense.
\end{theorem}

\begin{proof}
(1) Let $\ca$ denote  the full subcategory of $\ck$ on all $n. G =\coprod\limits_{n} G$, $n\in \N$. By Remark \ref{R:dense} (1), we are to prove that given objects $K$ and $L$ and a cocone of the canonical diagram $\ca/K\to \ck$ with codomain $L$, notation
$$
\xymatrix@R=.21pc{
&n\otimes G \ar[r]^{\quad f\quad } & K&\\
\ar@{-}[rrr]&&&\\
&n\otimes G \ar[r]_{f'} & L&
}\qquad (n\in \N)
$$
then (2a) in \ref{R:dense} holds: there exists a  morphism
$$
k\colon K\to L \quad \mbox{with} \quad f' = kf\quad \mbox{for all $f\colon n. G\to K$}.
$$
Property (b) in \ref{R:dense}  follows from (1).
Given $\bar k$ with $k\cdot f \leq \bar k \cdot f$ for all $f\colon G\to K$ (we can restrict ourselves to $n=1$), it follows  that $k\cdot c_K \leq \bar k \cdot c_K$, use Remark~\ref{R}(4). Thus $k\leq \bar k$, since $c_K$ is a subregular  epimorphism.

 \vskip 1mm
 (2) We extend $(-)^\prime$ to all finite tensors of $G$. Given
a finite poset $P$ and a morphism
 $f\colon P\otimes G\to K$ we define $f'\colon P\otimes G\to L$ as follows. We use the  coinserter of Remark~\ref{R2} (2):  
 \begin{equation*}\tag{a1}
\xymatrix@C=4pc{
P^{(2)} . G \ar @<1ex>[r]^{\pi_1. G}
\ar@<-1ex> [r]_{\pi_0. G} & \ |P|\otimes G \ 
\ar[r]^{k_ p\otimes G}
  \ar @{-->}[dr]_{f\cdot k_p\otimes G} & P\otimes G
\ar[d]^{f}\\
&& K
}
\end{equation*}
For $f\cdot k_p\otimes G$ $^{\ast)}$ \footnote{$^{\ast)}$ 
For simplifying our notation we assume that $\otimes$ binds stronger that composition. Thus $f\cdot (k_p \otimes G)$ is written as $f\cdot k_p\otimes G$.}
we already have, since $|P| \otimes G= |P|.G$ 
the corresponding morphism
$(f\cdot k_p\otimes G)^\prime \colon |P| \times G\to L$. We define $f'$ by verifying the following inequality
\begin{equation*}\tag{a2}
(f\cdot k_p\otimes G)^\prime \cdot \pi_0\otimes G \leq (f\cdot k_p \otimes G)^\prime \cdot \pi_1\otimes G\,.
\end{equation*}
Thus, $(f\cdot k_p\otimes G)^\prime$ factorizes uniquely through the coinserter   $k_p\otimes G$. Then $f'\colon P\otimes G \to L$ is defined as that factorization:
\begin{equation*}\tag{a3}
\xymatrix{
P\otimes G \ar [r]^{f'} & L\\
|P|\otimes G \ar[u]^{k_p\otimes G} \ar[ur]_{(f\cdot k_p\otimes G)^\prime}
&
}
\end{equation*}

To verify \thetag{a2}, use that
since $(-)^\prime$ is a cocone, it is monotone ($f\leq g$ implies $f'\leq g'$) and for every morphism $d\colon n. G\to m. G$ of $\ca$  we have the following implication

{\vbox{\begin{equation*}\tag{a4}
\xymatrix@C=.5pc{
n. G  \ar[rd]_{f} \ar[rr]^{d} \ar@{}[drr]|{\leq}&& m. G \ar[ld]^{g}\\
& K &
}\quad \quad \quad 
\xymatrix@C=.5pc{
n. G \ar[dr]_{f'} \ar[rr]^{d} \ar@{}[drr]|{\leq}&& m. G\ar[dl]^{ g'}\\
& L &
}
\end{equation*}
\quad \vskip -1.5cm \hskip 5.6cm $\Rightarrow$
}}
\vskip0.91cm
Applying this to $d=\pi_k . G$, for $k=0,1$, we get
\begin{align*}
(f\cdot k_p\otimes G)^\prime \cdot (\pi_k . G) & = 
(f\cdot k_p\otimes G\cdot \pi_k\otimes G)^\prime\\
&= (f\cdot [k_p \cdot \pi_k]\otimes G)^\prime\,.
\end{align*}
Since $(-)^\prime$ is monotone, and $k_p\cdot \pi_0\leq k_p \cdot \pi_1$ implies $[k_p\cdot \pi_0]\otimes G\leq [k_p \cdot \pi_1] \otimes G$ (because $-\otimes G$ is locally monotone by Proposition \ref{P}), we get
\begin{align*}
(f\cdot k_p\otimes G)^\prime \cdot \pi_0\otimes G &= (f\cdot [k_p\cdot \pi_0]\otimes G)^\prime\\
&\leq (f\cdot [k_p\cdot \pi_1]\otimes G)^\prime\\
&= (f\cdot k_p\otimes G)^\prime \cdot \pi_1\otimes G\,.
\end{align*} 

\vskip 1mm
(3) Given a monotone map $u\colon P\to Q$ we prove the following implication

\quad \hskip 1cm
$
\xymatrix@C=.5pc{
P\otimes G  \ar[rd]_{f} \ar[rr]^{u\otimes G} \ar@{}[drr]|{\circlearrowright}&& Q\otimes G \ar[ld]^{g}\\
& K &
}$
\quad\vskip-1cm \hskip 5.5cm$\Rightarrow $
\quad \vskip-1.4cm\hskip6.5cm
$\xymatrix@C=.5pc{
P\otimes G \ar[dr]_{f'} \ar[rr]^{u\otimes G}\ar@{}[drr]|{\circlearrowright}&& Q\otimes G\ar[dl]^{ g'}\\
& L &
}
$
\\
Since $k_P$ and $k_Q$ are carried by identity maps, we have
$$
u\cdot k_P= k_Q \cdot |u| \colon |P|\otimes Q \to Q\otimes G\,.
$$
 Consider the following diagram
 $$
\xymatrix{
|P|.G= |P|\otimes G \ar[rr]^{k_P\otimes G}
\ar[dd]_{|u|. G}
\ar[dr]^{(f\cdot k_P\otimes G)'}&&
P\otimes Q \ar[dd]^{u\otimes G}
\ar[ld]_{f'}\\
&L &\\
|Q|.G=|Q|\otimes G \ar[ur]_{(g\cdot k_Q \otimes G)'}
\ar[rr]_{k_Q\otimes G}&& Q\otimes G \ar[ul]^{g'}
}
$$
The square commutes 
because $k_P$ and $k_Q$ are carried by identity maps,
 and the upper and lower triangles commute by \thetag{a3}. The left-hand triangle commutes since \thetag{a4} yields for $d=|u|\otimes G$\ the following equality
\begin{align*}
\big(g\cdot k_Q \otimes G\big)^\prime \cdot \big(u. G \big)&= \big(g\cdot k_Q \otimes G\cdot |u|\otimes G\big)^\prime\\
&= \big(g\cdot u \otimes G\cdot k_P \otimes G\big)^\prime\\
&=\big(f\cdot k_P \otimes G\big)^\prime\,.
\end{align*}
Thus the right-hand triangle commutes when precomposed by $k_P\otimes G$. Since $k_P\otimes G$ is an epimorphism (see Proposition \ref{P}), $f'=g'\cdot u\otimes G'$.

\vskip 2mm 
(4) It follows  from Remark~\ref{R}(3) that the canonical  morphism $c_K$
fulfils, for every finite subposet $m\colon M\hookrightarrow \ck(G,K)$, that the following triangle commutes:
\begin{equation*}\tag{a5}
\xymatrix@C=4pc{
M\otimes G \ar[d]_{m\otimes G} \ar[dr]^{\hat m}&\\
\ck (G,K) \otimes G \ar[r]_{c_K} & K
}
\end{equation*}
 Since $\ck(G,K)$ is a directed colimit of its finite subposets in $\Pos$, we conclude
(from Proposition \ref{P} again)
 that all $m\otimes G$ form a colimit cocone in $\ck$. The morphisms $\hat m'\colon M\otimes G \to L$ form a cocone of that diagram.  Indeed, given $M_1 \subseteq M_2\subseteq \ck(G, K)$ we denote by $u\colon M_1\to M_2$  the inclusion map and derive $\hat m_1' =  \hat m_2' \cdot u\otimes G$ from (2):\\
$
\xymatrix@R=4pc@C=1pc{
M_1\otimes G \ar[rr]^{u\otimes G} 
\ar[ddr]_{\hat m_1}
\ar[dr]^{m_1\otimes G}
\ar@{}[drr]|{\circlearrowright}&&
M_2\otimes  G \ar[dl]_{m_2\otimes G}
\ar[ddl]^{\hat m_2}\\
& \ck (G,K) \otimes G \ar[d]^{c_K} &\\
& K &
}
$
\quad\vskip-3cm \hskip 7cm$\Rightarrow $
\quad \vskip -2.2cm \hskip 8cm
$
\xymatrix@R=9pc{
M_1\ar[rr]^{u\otimes G} \ar[dr]_{\hat m_1'}
\ar@{}[drr]|{\circlearrowright}&
 & M_2 \ar[dl]^{\hat m_2'}\\
& L &
}
$
\\
Thus there exists a unique morphism
$$ q\colon \ck (G,K)\otimes G\to L
$$
making the following triangles commutative
\begin{equation*}\tag{a6}
\xymatrix@C=4pc@R=4pc{
M\otimes G \ar [d]_{m\otimes G} \ar[dr]^{(\hat m)'} &\\
\ck (G,K)\otimes G \ar[r]_{q} & L
} 
\end{equation*}
{for all finite subposets  $m\colon M\hookrightarrow \ck(G,K)$}. By Lemma \ref{L:ca} we can
express $c_K$ as a coinserter of a parallel pair $u_0$, $u_1$:
$$
\xymatrix{
U \ar @<1ex>[r]^{u_1\qquad\quad}
\ar@<-1ex> [r]_{u_0\qquad\quad} & \ \ck (G,K)\otimes G \ 
\ar[r]^{\quad\quad c_K}
  \ar [d]_{q} &K
\ar @{-->}[dl]^{k}\\
&L&
}$$
In the next point we prove that $qu_0\leq qu_1$. Thus $q$ factorizes as $k\cdot c_K$. Then $k$ is the desired morphism: we prove
$$
f' = k\cdot f \quad \mbox{for all \quad} f\colon n\otimes G\to K\,.
$$
It is sufficient to verify this for $n=1$ since for general $f= [f_0, \dots , f_n]$ we have $f' = [f_0', \dots f_n']$ (apply \thetag{a4} to the coproduct injections $u\colon G\to n\otimes G$) and thus $f_i' = k\cdot f_i$ imply $f' = k\cdot f$.

Given $f\colon G\to K$ we have the subposet $m \colon\{f\}\hookrightarrow \ck(G,K)$ with $\hat m=f$, for which \thetag{a6} yields
$$
q\cdot m\otimes G=f'\,.
$$
Since $q=k\cdot c_K$ and $c_K\cdot m\otimes G =\hat m=f$ by \thetag{a5}, we get
$$
k\cdot f = f'\,.
$$
\vskip 2mm
(5) It remains to verify $q\cdot u_0\leq q\cdot u_1$. Since $G$ is a generator,  this is equivalent to 
$$
q\cdot u_0\cdot r\leq  q\cdot u_1\cdot r \quad \mbox{for all} \quad r\colon G\to U\,.
$$
For the poset $P=\ck(G,K)$ we have the morphism $k\colon |P| \to P$ carried by the identity map. It is a subregular epimorphism in $\Pos$ (Example \ref{E:subreg}), thus, due to Proposition \ref{P}  $k\otimes G\colon |P| \otimes G \to P\otimes G$ is a subregular epimorphism in $\ck$. Since $G$ is a subregular projective, we have factorizations
$$
u_i r = (k\otimes G) u_i' \quad \mbox{for}\quad u_i' \colon G\to |P| \otimes G \qquad (i=0,1)\,.
$$
Moreover, $G$ is abstractly finite, thus there exists a finite subset $m \colon M\to |P|$ with factorizations
$$
u_i'= (m\otimes G) .v_i \quad \mbox{for}\quad v_i\colon G\to M.G \qquad (i=0,1)\,.
$$
 In other words: $v_i$ is a factorization of $u_ir$ through $(km)\otimes G$:
$$
\xymatrix@C=4pc{
G \ar @<1ex>[r]^{ v_1\qquad\quad}
\ar@<-1ex> [r]_{v_0\qquad\quad}
\ar[d]_{r} & \ M \otimes G \ 
\ar[dr]^{\quad\quad \widehat {km}}
  \ar [d]^{(km)\otimes G} &\\
U \ar @<1ex>[r]^{ u_1\qquad\quad}
\ar@<-1ex> [r]_{ u_0\qquad\quad}& \ck (G,L)\otimes G \ar[d]^{q} \ar[r]_{\qquad c_K}& K\\
& L &
}$$
We know that $c_k u_0 r \leq c_k u_1 r$, and this by \thetag{a5} impllies 
$$
\widehat{km} v_0 \leq \widehat{km} v_1\,.
$$
From \thetag{a4} we then obtain $\widehat{km}' v_0\leq \widehat{km}' v_1$, thus,
$$
qu_0 r = \widehat{km}' v_0 \leq \widehat{km}' v_1 = qu_1 r\,.
$$
\end{proof}

\vskip 1mm
\begin{remark}\label{R:ref}
(1) A full subcategory $\ck$ of $\cl$ is  called reflective if the embedding $\ck\hookrightarrow \cl$ has a left adjoint. Suppose that $\cl$ has weighted limits and colimits, then so does every full reflective subcategory (\cite{K}, Section 3.5). This is in particular the case if $\cl$ is the functor category $[\ca^{\op}, \Pos]$ for $\ca$ small (\cite{K}, Section 3.3).

\vskip 1mm
(2)
Let $G$ be an object with copowers in a category with reflexive coinserters. Then every parallel pair $f$, $g \colon \coprod\limits_{A} G \to \coprod\limits_{B} G$ has a conical coequalizer. Indeed for $X= \coprod\limits_{A} G + \coprod\limits_{A} G + \coprod\limits_{B}G$ consider the reflexive pair
$
[f,g, \id] , [g,f,\id] \colon X \to \coprod\limits_{B}G\,.
$
Its coinserter is precisely a conical coequalizer of $f$ and $g$.

\end{remark}

\begin{prop}\label{P:lim}

Let $G$ be  an abstractly  finite, subregularly projective  strong  generator 
in $\ck$. If $\ck$ has 
subkernel pairs and reflexive coinserters, then it has weighted limits and weighted colimits.
\end{prop}

\begin{proof}
(1) Let $\ca$ be the full subcategory of finite copowers $n. G$ which we know is dense by Theorem \ref{P1}. This implies that the functor $E$ 
(Notation \ref{N:str}) is full and faithful. In fact, fullness is precisely (2a) in Definition~\ref{R:dense}. To see the faithfulness:
$$
Ef_0 \leq Ef_1 \quad \mbox{implies}\quad f_0\leq f_1\,,
$$
use (2b) in that definition.

Thus, all we need to prove is that $E$ has a left adjoint (see the last step of our proof): then we apply Remark~\ref{R:ref}.

\vskip 1mm
(2) Since $G$, being abstractly finite, has tensors, all copowers ${M}. G$ exist and are conical, i.e., given $u_0$, $u_1\colon {M} . G \to X$  then $u_0\leq u_1$ iff this holds when precomposed by every coproduct injection. 

\vskip 1mm
(3) For every object $H$ of $[\ca^{\op}, \Pos]$ we construct a conical  diagram $D_H \colon \ce_K\to \ck$ `of elements' of $H$.

The objects of $\ce_H$ are pairs $(A,x)$ where $A\in \ca$ and $x\in HA$. For a pair $(A,x)$ and $(B, y)$ of objects morphisms $f\colon (A,x)\to (B,y)$ are those morphisms $f\colon A\to B$ of $\ca$ with $Hf(y)=x$. They are ordered as in $\ca (A,B)$. And we define
$$
D_H\colon \ce_H \to \ck\,, \quad (A,x) \mapsto A\,.
$$
From the conical copowers of $G$ it follows  that $D_H$ has a conical colimit.  We use the standard construction of conical colimits via conical coproducts and conical  coequalizers,  completely analogous to the non-enriched case (\cite{McL}, Thm. V.2.1): put
$$
X=\coprod\limits_{(A,x)} A \quad \mbox{and}\quad Y= \coprod\limits_{f\colon (A,x) \to (B,y)} B
$$
where $X$ is a coproduct ranging over  objects of  $\ce_H$ and $Y$ is one ranging over  morphisms of  $\ce_H$. Let us denote the coproduct injections of $X$ by
$$
i_x \colon A \to X \qquad (x\in HA)\,.
$$
The conical colimit $C$ of $D_H$ is then obtained as the following coequalizer
(using Remark~\ref{R:ref}(2)):
$$
\xymatrix@1{ 
Y \ar@<1ex>[r]^{p}
\ar @<-1ex>[r]_{q} \ & \ \ X\  \ar[r]^{c} &\ C
}
$$
where the components of $p$ and $q$ corresponding to $f\colon (A,x) \to (B,y)$ are $i_x$ and $i_y\cdot f$, resp. The colimit cocone is $ci_x \colon (A,x) \to C$.

\vskip 1mm
(4) We thus can define a functor
$$
L\colon [\ca^{\op}, \Pos] \to \ck
$$
by assigning to every object $H$ the colimit
$$ LH = \colim D_H\,.
$$
We verify that $L$ is a left adjoint of $E$. Now consider an object $K\in \ck$ and the corresponding object $EK = \ck(-, K)\big/ \ca^{\op}$. To give a natural transformation from $H$ to $EK$ is precisely to give a cocone of $D_H$ with codomain $K$. We thus obtain the desired natural order-isomorphism
$$
\xymatrix@R=.21pc{
&H \ar[r] & EK&\\
\ar@{-}[rrr]&&&\\
&LH \ar[r] &K&
}
$$
proving that $L$ is left adjoint to $E$.
\end{proof}

\section{Varieties as abstract categories}\label{sec3} 

\begin{nota}\label{no3.1}
Let $\Sigma = (\Sigma)_{n\in \mathbb{N}}$ be a signature. The category
$$
\Sigma\mbox{-}\Alg
$$
has as objects ordered $\Sigma$-algebras: posets with a structure of a $\Sigma$-algebra whose operations are monotone. Morphisms are the monotone homomorphisms.
\end{nota}

A \textit{variety} of ordered algebras is a full subcategory of  $\Sigma$-$\Alg$  specified by a set of  inequations between terms.

\begin{example}
(1) Ordered monoids form a variety  of $\Sigma$-algebras for $\Sigma =\{\circ, e\}$ specified by the usual monoid equations.

(2) Ordered monoids with the least element   $e$ form the subvariety specified by the inequation $e\leq x$.
\end{example}

For a given algebra $A$ a \textit{subalgebra} is represented  by a homomorphism $m\colon B\to A$ carried by an order-embedding: $x\leq y$ in $B$ iff $m(x) \leq m(y)$ in $A$. A \textit{quotient  algebra} is represented by a surjective monotone homomorphism $c\colon A\to C$. The following result was sketched by Bloom \cite{B1}, a detailed proof  can be found in \cite {ADV}.

\begin{birk}\label{BVT} 
A full subcategory of  $\Sigma\mbox{-}\Alg$ is a variety (i.e., can be presented by a set of inequations) iff it is closed under products, subalgebras and quotient algebras.
\end{birk}

Homomorphic images in varieties are precisely the subregular quotients:
\begin{prop}\label{P:surj}
Subregular epimorphisms in a variety of ordered algebras are precisely the surjective homomorphisms.
\end{prop}

\begin{proof}
(1) If $h\colon A\to B$ is subregular, say, a coinserter of $u_0$, $u_1 \colon U\to A$, then it is surjective. Indeed, the subalgebra $B'$ of $B$ on $h[A]$ lies in our variety $\cv$, and $h$ restricts to a morphism $h'\colon A\to B'$ of $\cv$. 
It clearly satisfies $h'u_0 \leq h'u_1$, thus, there exists $f\colon B\to B'$ with $h'=f.h$. The inclusion $i\colon B'\to B$ fulfils
$$
(fi) h' = fh = h'\,,
$$
thus  $fi=\id$ since $h'$ is surjective. From the universal property of $h$ we deduce, since $(if)h=ih'=h$, that $if=\id$. Thus $i= f^{-1}$, proving that $B'=B$, as stated.

(2) If $h\colon A\to B$ is surjective in $\cv$, let $E$ be the  subalgebra of $A\times A$ on all $(x,y)$ with $h(x) \leq h(y)$. (This is closed under operations since $h$ is a nonexpanding homomorphism.) From $A\in \cv$ we conclude $A\times A\in \cv$, hence, $E\in \cv$. The restricted projections $u_0, u_1\colon E \to A$ are  morphisms of $\cv$ with $u_0\leq u_1$, and they form clearly a reflexive pair (since $E$ contains the diagonal of $A$). It is easy to see that $h$ is a coinserter of $u_0$, $u_1$.
\end{proof}

The concept of effective projective (Definition~\ref{D:eff}) has the following enriched  variant:

\begin{defi}\label{de3.3}
An object whose hom-functor (into $\Pos$) preserves coinserters of reflexive pairs is called  a \textit{subeffective projective}.
\end{defi}

\begin{example}\label{E:var} 
(1) In every variety $\mathcal{V}$ the free algebra $G$ on one generator is a subeffective projective. Indeed, its hom-functor is naturally isomorphic to the forgetful functor $U\colon \cv\to \Pos$.

As proved in \cite{ADV}, $U$ is a monadic functor  preserving reflexive coinserters.

\vskip 1mm
(2)
Moreover, $G$ is finitely presentable in the enriched sense since $U$ is finitary. And it is a subregular  generator. Indeed, the universal property of $P\otimes G$ implies for every poset $P$ that
$$
P\otimes G \quad \mbox{is a free algebra on $P$ in $\mathcal{V}$}\,.
$$
For every algebra $K$ in $\mathcal{V}$ the canonical homomorphism
$$
c_K \colon \mathcal{V}(G,K)\otimes G\to K
$$
is the unique extension of $\id_K$ to a homomorphism from the free algebra on 
$UK$ to $K$, and this  is a subregular epimorphism by Proposition \ref{P:surj}.

(3)
 Finally, $G$ is a strong generator (Definition \ref{D:str}). Indeed, 
$$
U\colon \cv \to \Pos
$$ 
is conservative, thus so is $E\colon \cv \to \Pos^{\cv(G,G)^{\op}}$ because $U=V\cdot E$ for the  forgetful $V\colon \Pos^{\cv(G,G)^{\op}}\to \Pos$.
\end{example}

\begin{remark}\label{R4.6}
Every subeffective projective is a subregular projective, provided that  subkernel pairs exist. Indeed, every subregular  epimorphism is the coinserter of its subkernel pair (which is reflexive).

\end{remark}

\begin{theorem}\label{T:main}
A category with  reflexive coinserters is equivalent to a variety of ordered algebras iff it has a subregular   generator which is an abstractly  finite subeffective projective.
\end{theorem}

\begin{proof}
In view of the previous example we need to prove  only the sufficiency:
 if $\ck$ has subkernel pairs and reflexive coinserters and a generator $G$ with the above properties, then it is equivalent to a variety. Recall from Proposition \ref{P:lim} that $\ck$ has weighted limits and colimits.

\vskip 1mm
(1)
$\ck$ has a factorization system with $\ce$ all subregular epimorphisms (coinserters of  some pairs) and $\mathcal{M}$ all embeddings $m$ (Definition \ref{D:emb}). Indeed, let $f\colon X\to Y$ be a morphism
and choose a subkernel 
 pair $p_0$, $p_1 \colon P\to X$.  It is clearly reflexive.  Let $c\colon X\to Z$ be a coinserter of $p_0$, $ p_1$ and $m\colon Z\to Y$ the unique morphism with $f=m\cdot c$. The proof  that $m$ is  an embedding is completely analogous to point (1) of the proof of Theorem~\ref{P1}.

\vskip 1mm
(2)
We now define a full embedding $E\colon \ck \to \Sigma\mbox{-}\Alg$ for the following signature $\Sigma$:
$$
\Sigma_n = \ck_0 (G, n . G) \qquad (n\in \N)\,.
$$
That is, an $n$-ary operation symbol is precisely a morphism $\sigma \colon G \to n. G$ of $\ck_0$, the ordinary category underlying $\ck$.

The algebra $EK$ assigned to an object $K$ has the underlying poset $\ck(G,K)$. Given  an $n$-ary operation $\sigma$ and an $n$-tuple $(f_0, \dots , f_{n-1}) $ in $\ck(G,K)$ we form the morphism $[f_i] \colon n. G\to K$, and  define the result of $\sigma_{EK}$ in our $n$-tuple as the following composite
$$
\sigma_{EK}(f_i) \equiv G \xrightarrow{\ \sigma\ } n . G \xrightarrow{\ [f_i]\ } K\,.
$$
To every morphism $h\colon K \to L$ we assign the homomorphism
$$
Eh = \ck (G,h)
$$
of post-composition with $h$. Then $Eh$ is clearly monotone.
 Preservation of $\sigma\colon G \to n . G$ is clear:
 $$
 Eh\big(\sigma_{EK} (f_i)\big) = h\cdot [f_i]\cdot \sigma = [h\cdot f_i] \cdot \sigma = \sigma_{EL} (h\cdot f_i)\,.
$$

\vskip 1mm
(2a)
$E$ is a fully faithful functor.  To prove  that it is full, let $k\colon EK\to EL$ be a homomorphism. By Theorem~\ref{P1} it is sufficient to verify the naturality of the following transformation
$$
\xymatrix@R=.21pc{
&n .G \ar[r]^{[f_i]} & K&\\
\ar@{-}[rrr]&&&\\
&n . G  \ar[r]^{\ k(f_i)} &L&
}
$$
That is, we need to prove the following implication for all morphisms $u\colon n . G\to m. G$ ($
n, m\in \mathbb{N}$):\\
\indent\qquad\quad 
$
\xymatrix@C=.5pc{
n. G  \ar[rd]_{[f_i]} \ar[rr]^{u} \ar@{}[drr]|{\circlearrowright}&& m. G \ar[ld]^{[g_j]}\\
& K &
}
$\quad\vskip-1cm \hskip 5cm$\Rightarrow$
\quad \vskip-1cm\hskip6cm
$
\xymatrix@C=.5pc{
n. G  \ar[rd]_{[k(f_i)]} \ar[rr]^{u} \ar@{}[drr]|{\circlearrowright}&& m. G \ar[ld]^{[k(g_j)]}\\
& L &
}
$

From this it follows that there exists $h\colon K\to L$ with $h\cdot [f_i]= [ k(f_i)]$ for all $[f_i] \colon n . G\to K$. The case $n=1$ then yields
$$
h\cdot f= k(f) \quad \mbox{for all} \quad f\colon G\to K
$$
in other words, $Eh =k$, as desired.

The above implication is clear if $n=1$: here $u$ is an $m$-ary operation symbol in $\Sigma$  and $f_0 =[g_j]\cdot u = u_{EK}(g_j)$. Since $k$ is a homomorphism, we deduce
$$
k(f_0) = u_{EL} \big(k(g_j)\big) = \big[k(g_j)\big]\cdot u\,.
$$
For $n>1$ that implication follows by considering the $n$ components of $u$ separately.


To prove that $E$ is faithful, that is given $k_0$, $k_1\colon K\to L$ with $Ek_0 \leq Ek_1$, we conclude $k_0\leq k_1$, use that fact that 
by Lemma \ref{L:ca}
the canonical morphism $c_K \colon \coprod\limits_{\ck(G,K)} G\to K$ is a subregular  epimorphism. 

\vskip 1mm
(2b) $E$ preserves limits, filtered colimits, and reflexive coinserters. In fact, if $U\colon \Sigma\mbox{-}\Alg\to \Pos$ denotes the forgetful functor, then
$$
U\cdot E= \ck (G, -) \colon \ck \to \Pos\,.
$$
$U$ creates limits, filtered colimits and reflexive coinserters by Example~\ref{E:var}. Since $\ck(G,-)$ preserves all those three types of constructions, so does $E$.

\vskip 1mm
(3) $\ck$ is equivalent to a variety. For that denote by $\bar \ck$ the closure of the image of $E$ under isomorphism in $\Sigma\mbox{-}\Alg$. From (2a) we know that $\ck$ is equivalent to $\bar \ck$. We now use the Birkhoff Variety Theorem~\ref{BVT} to prove that $\bar \ck$ is a variety.

\vskip 1mm
 (3a) $\bar \ck$  is closed under products because $\ck$ has products by Proposition \ref{P:lim} and $E$ preserves them.
 
 \vskip 1mm
 (3b) $\bar \ck$  is closed under subalgebras. It is sufficient to prove this for finitely generated subalgebras. Indeed, $\bar \ck$ is closed under directed colimits (since $\ck$ has them and $E$ preserves them), and every subalgebra is a directed colimit of finitely generated  subalgebras.
 
 Thus our task is, for every object $K$ of $\ck$ and every finite subposet $m\colon M\hookrightarrow \ck(G,K)$, to prove that the least subalgebra of $EK$ containing $m$ lies in $\bar \ck$. Factorize $\hat  m\colon M\otimes G\to K$ as a  subregular epimorphism $c$ followed  by an embedding $u$ :
 $$
\xymatrix{
 M\otimes G\ar[rr]^{\hat m} \ar@{->>}[rd]_{c} && K\\
& L\ar@{>->}[ur]_{u}&
}
$$
We will prove that $Eu$ represents the least subalgebra  containing $m$. That this is a subalgebra of $EK$ is clear:
we know that $Eu$ is a monotone homomorphism, and we have: $Eu(x_0)\leq Eu(x_1)$ implies $x_0\leq x_1$ because $u$ is an embedding with $u\cdot x_0\leq u\cdot x_1$.

We verify that every subalgebra $B$ of $K$ containing $M$:
$$
M\subseteq UB \subseteq \ck(G,K) 
$$
also contains the image of $Eu$. That is:
$$
u\cdot z\in B\quad \mbox{for all}\quad z\colon  G\to L\,.
$$
We know that
since $G$ is a subeffective projective, 
 $\ck(G, c)$ is surjective, thus, $z$ factorizes as $z=c\cdot z'$ for some $z'\colon G\to M\otimes G$. We next use that $\ck(G,-)$ preserves the following reflexive coinserter
$$
\xymatrix@1@C=4pc{
M^{(2)}. G \ar@<.8ex> [r]^{\pi_0. G}
\ar@<-.8ex>[r]_{\pi_1. G}\  &\  |M|. G  \ar[r]^{k_M \otimes G}\  &\  M\otimes G
}
$$
of Example~\ref{R2}. Thus, $z'\colon G\to M\otimes G$ factorizes through $k_M \otimes G$ via an operation
$$
\sigma \colon G \to |M| \otimes G
$$
of arity card $|M|$:
 $$
\xymatrix@C=1pc{
&&&& \\
|M|. G \ar@/_-2pc/[rrrr]^{\hat m_0}\ar[rr]^{k_M\otimes G} &&M\otimes G \ar[rr]^{\hat m} \ar[dr]^{c} && K\\
& G\ar[ul]^{\sigma} \ar[ur]_{z'} \ar[rr]_{z}
&&  L \ar[ur]_{u}
}
$$
The morphism $m_0 = m\cdot k_M \colon |M| \to \ck(G,K)$ fulfils $\hat m_0 =\hat m \cdot k_M\otimes G$ by Remark~\ref{R}(2), therefore,
$$
u\cdot z = \hat m_0 \cdot \sigma =\sigma_{EK}(m_0)\,.
$$
Since $B$ is closed under $\sigma_{EK}$ and  contains $m_0$, this proves $u\cdot z\in B$.

\vskip 1mm
 (3c) $\bar {\ck}$ is closed under quotient algebras. Let $e\colon EK \to A$ be a surjective homomorphism. Form its subkernel  pair  $u_0$, $u_1\colon Z\to UEK$ in $\Pos$. Then $Z$ is closed in $UEK \times UEK \simeq UE(K\times K)$ under operations.
Indeed, let $\sigma$ be an  $n$-ary operation. If an $n$-tuple $f_0, \dots , f_{n-1} \colon G\to K\times K$ lies in $Z$, more precisely, it factorizes through $\langle \pi_0, \pi_1\rangle \colon Z \hookrightarrow UE (K\times K)$, we have $f_i =\langle \pi_0, \pi_1\rangle. g_i$ from which we deduce that $\sigma_{E(K\times K} (f_i)$ lies in $Z$:
$$
\sigma_{E(K\times K} (f_i) = [f_i] .\sigma =\langle u_0, u_1\rangle \cdot  [g_i]\cdot \sigma\,.
$$
 By (3a) and (3b) there exists a subobject $m=[m_0, m_1] \colon M \to K\times K$ with $Em_0$, $Em_1$ forming the prekernel pair of $e$. Since  $m_0$, $m_1$ is a reflexive pair, $E$ preserves its coinserter $k\colon K\to L$. Thus $A\simeq EL$ lies in $\bar{\ck}$.
 
\end{proof}

\section{Varieties as Free Completions} 

The aim of this section is to prove a parallel result to Theorem~\ref{T:sif}: varieties of ordered algebras are precisely  the free completions of $\ct^{\op}$ under sifted (weighted) colimits, where $\ct$ ranges over  discrete Lawvere theories. This latter concept was introduced by Power \cite{P} for  algebras of countable arities. The corresponding finitary variant uses $\cn$ of Notation~\ref{N:N}, now considered  as a trivially enriched category: all hom-sets are discrete.

Recall that all categories, functors etc. are enriched over $\Pos$.

\begin{defi}[{\cite{P}}]\label{5.1}
 A (finitary) \textit{discrete Lawvere theory} is a small enriched category $\ct$  with specified finite products together with a functor $I\colon \cn^{\op} \to \ct$ which is identity on objects and strictly  preserves finite products.
 \end{defi}
 
 The category $\Mod \ct$ of models is now defined analogously to the ordinary case: it consists of (enriched) functors $A\colon \ct \to \Pos$ preserving finite products and  natural transformations.
 
 
\begin{example} The discrete Lawvere theory $\ct_{\cv}$ associated to a variety of ordered algebras  $\cv$ has as objects natural numbers, the hom-poset $\ct_{\cv}(n, 1)$ is the underlying poset of the free algebra of $\cv$ on $n$, and $\ct_{\cv} (n, k) = \ct_{\cv}(n, 1)^k$. Thus $n$ is the power of $1$: the projection $\pi_i \colon n\to 1$ corresponds to $i$ as an element of the free algebra on $n$.
\end{example}

Every algebra $A$ of $\cv$ defines a model $\widehat A$ of $\ct_{\cv}$: to each $n$ it assigns the underlying poset of $A^n$. To every morphism $f\in \ct_{\cv} (n, 1)$ it assigns the monotone  map $\widehat f\colon A^n\to A$ which, given an $n$-tuple $h\colon n\to A$, extends it to the unique homomorphism $h^\sharp  \colon \ct_{\cv}(n, 1) \to A$ and yields
$$
\widehat f (h) = h^\sharp (f)\,.
$$

We now turn to sifted colimits in the enriched setting (cf.~Definition~\ref{D:sif}) following  the dissertation of Bourke \cite{B} (where the base category of categories was considered) and the paper \cite{KV} in which the appropriate adaptation  to $\Pos$ was made explicit:
 
\begin{defi}[{\cite{B}}]
 A weight $W\colon \cd^{\op} \to \Pos$ is called \textit{sifted\/}  if colimits of diagrams in $\Pos$  weighted by $W$ commute with finite products: given diagrams $D_1$, $D_2\colon \cd \to \Pos$, the canonical morphism $\colim_W(D_1\times D_2) \to (\colim_W D_1) \times (\colim_W D_2)$ is an isomorphism.

 Weighted colimits in a category are called \textit{sifted colimits} if the weight is sifted.
 \end{defi}
 
 \begin{example}\label{E:fil}  
 (1) Filtered colimits (those where $\cd$ has a filtered underlying category) are sifted.
 
 (2) Reflexive coinserters are sifted colimits \cite{ADV}.
 
 (3) Also split coequalizers are sifted colimit. This follows from
  \cite{LR} 1.3, because split coequalizers are absolute colimits (i.e. every functor preserves split coequalizers). This can be directly verified as follows. Let
$$ 
	\xymatrix@=4pc{
		&  
		A\ar@<0.5ex>[r]^{d_0}
		\ar@<-0.5ex>[r]_{d_1}& B  \ar[r]^{e} & C
	}
	$$
be a coequalizer split by $t:B\to A$ and $s:C\to B$. That is $es=\id_C$, $d_0t=\id_B$ and $d_1t=se$. We have
$$
d_1td_1=sed_1=sed_0=d_1td_0.
$$
Conversely, having $t\colon B\to A$ such that $d_0t=\id_B$ and $d_1 t d_0=d_1 t d_1$, we get
a unique $s\colon C\to B$ such that $se=d_1t$.

Hence $C$ is a colimit of the filtered diagram consisting of $d_0,d_1$ and $t$.

\end{example}

 \begin{remark}\label{R5.5}
(1) Analogously to the ordinary categories (see Example~\ref{R:sif}), the first  two cases  above are in a way exhaustive. For example, an endofunctor of $\Pos$ preserves sifted colimits iff it preserves filtered colimits and reflexive coinserters (\cite{B}, 8.45).

(2)
For every variety $\ck$ the forgetful functor $U\colon \ck\to \Pos$ preserves 
(indeed: creates) filtered colimits. This is easy to verify. 
\end{remark}

\begin{remark}\label{monad}
Following \cite[Theorem 4.5]{ADV}, varieties of ordered algebras are, up to concrete isomorphism,  precisely categories
of algebras $\Pos^{T}$  for enriched monads  $T$ on $\Pos$ preserving sifted colimits. Moreover, these  are up to equivalence precisely the categories of models of discrete Lawvere theories: see Theorem 7.7 cf \cite{LR}.
(That result is more general, dealing with a closed category $\cv$ and a class $\phi$ of limits. Applying it to $\cv =\Pos$ and $\phi=$ finite products yields the above special case.)

The equivalence of op.cit. assigns to every variety $\cv$ the theory $\ct_{\cv}$: it  turns out that all  models of $\ct_{\cv}$ are naturally equivalent to $\widehat A$  for algebras $A\in \cv$.
\end{remark}

\begin{remark} 
Let $\ck$ be an enriched  category.
By its \textit{free completion under sifted colimits\/} is meant an enriched  category $\Sind \ck$ with sifted colimits containing $\ck$ as a full subcategory  having  the expected universal property: for every functor $F\colon \ck \to \cl$ where  $\cl$  has sifted colimits there exists an extension to $\Sind \ck$ preserving sifted colimits, unique up to natural isomorphism.

It follows that the functor category $[\ck, \cl]$ is equivalent (via the domain-restriction functor) to the full subcategory of $[\Sind \ck, \cl]$ formed by functors preserving sifted colimits.
\end{remark}
\begin{defi}
An object $A$ of $\ck$ is called 
\textit{perfectly presentable} if its hom-functor $\ck(A,-):\ck\to\Pos$ preserves sifted
colimits.
\end{defi}

\begin{remark}\label{R5.9}
By Example~\ref{E:fil} 
every perfectly presentable object is finitely presentable and a subeffective projective.
\end{remark}

\begin{lemma}\label{perfect}
Perfectly presentable objects are closed under finite coproducts and retracts.
\end{lemma}
\begin{proof}
The proof of the first claim is analogous to \cite{LR}, 5.6. The second statement is easy.
\end{proof}

\begin{prop}\label{perfect1}
Let $\ck$ be a variety of ordered algebras. The following properties of an arbitrary  algebra $A$ of $\ck$
are equivalent:
\begin{enumerate}
\item $A$ is perfectly presentable,
\item $A$ is finitely presentable and  a subregular projective, and
\item $A$ is a retract of a free algebra on a finite discrete poset.
\end{enumerate}
\end{prop}

\begin{proof}
(1) $\Rightarrow$ (2) follows from Example~\ref{E:fil}.

(2) $\Rightarrow$ (3): Since $A$ is finitely presentable,
there is a finite subposet $P$ of $A$ such that for the free-algebra functor $F\colon \Pos \to \ck$ a surjective homomorphism $e\colon FP\to A$ exists.
The canonical  coinserter of  $P$ (Example~\ref{R2}) is preserved by $F$, thus, $Fk_p \colon F|P| \to FP$ is also surjective. As in Example~\ref{E:subker} (2), we can prove that surjective homomorphisms are coinserters of their  subkernel pairs. Thus $e\cdot Fk_p \colon F|P| \to A$ is a subregular epimorphism. Since $A$ is a subregular projective, this implies that $e\cdot F k_p$ is a split epimorphism. Thus $A$ is a retract of $F|P|$.

(3) $\Rightarrow$ (1): Since by Remark~\ref{R5.5} the forgetful functor $U:\ck\to\Pos$ preserves sifted colimits,
its left adjoint $F:\Pos\to\ck$ preserves perfectly presentable objects. Thus (1) holds due to Remark \ref{R4.6}.
\end{proof}

The following theorem is due to Kurz and Velebil (\cite{KV}, 6.9 and 6.12). We present  a full proof because it is simpler than that in op.cit.

\begin{theorem}\label{perfect2}
Let $\ck$ be a variety of ordered algebras and $\cp$ its full subcategory on free algebras
on finite discrete posets. Then $\ck =\Sind \cp$. 
\end{theorem}

\begin{proof}
Following \cite[Proposition 4.2]{KS} and \ref{perfect1}, all we have to show is that $\ck$ is the closure
of $\cp$ under sifted colimits.
 
(1) A finite poset $P$ is a reflexive coinserter as in Example~\ref{R2}. This yields the following reflexive coinserter in $\ck$
$$ 
	\xymatrix@=4pc{
		&  
		FP^{(2)}\ar@<0.6ex>[r]^{F\pi_1}
		\ar@<-0.6ex>[r]_{F\pi_0}& F|P|  \ar[r]^{F k_p} & FP
	}
	$$
	with $FP$ in $\cp$.

(2) A free algebra on an arbitrary  poset $X$ is a filtered colimit of free algebras over finite posets. This follows from the fact that the free-algebra functor from $\Pos$ to $\ck$ preserves colimits:

(3) 
Finally, every algebra $A$ in $\ck$  is a split coequalizer of free algebras via its canonical
presentation
$$ 
	\xymatrix@=4pc{
		&  
		FUFUA\ar@<0.5ex>[r]^{\varepsilon_{FUA}}
		\ar@<-0.5ex>[r]_{FU\varepsilon_A}& FUA  \ar[r]^{\varepsilon_A} & A
	}
	$$
Following Example~\ref{E:fil}, this is a filtered colimit.	
\end{proof}

\begin{remark}\label{R5.13}
A concrete category over $\Pos$ is a  category $\ck$ together with a faithful (enriched) `forgetful' functor $U\colon \ck \to \Pos$. 

Given  concrete categories $(\ck, U)$ and $(\ck^\prime, U^\prime)$, they are (concretely) \textit{equivalent} if there exists an equivalence functor $E\colon\ck \to \ck^\prime$ with $U=U^\prime E$. Analogously, they are \textit{isomorphic} if $E$ is an isomorphism.
\end{remark}

\begin{theorem}\label{th5.14}
The following statements are equivalent for  an enriched category $\ck$ up to concrete equivalence:
\begin{enumerate}
\item $\ck$ is a variety of ordered algebras,
\item $\ck=\Sind  \ct^{\op}$  for a discrete
Lawvere theory $\ct$, and
\item $\ck = \Mod \ct$ for a discrete Lawvere theory $\ct$.
\end{enumerate}
\end{theorem}
\begin{proof}
(1)$\Leftrightarrow$(3) follows from \ref{monad}.

(1)$\Rightarrow$(2) follows from \ref{perfect2}.

(2)$\Rightarrow$(1):   Following \cite[Proposition 8.1]{KS}, $\Sind\ct^{\op}\subseteq\Mod\ct$. Due to  \cite[Theorem 7.7]{LR}, $\ct^{\op}$ is the category of free $\ct$-algebras on finite discrete posets. Due to \ref{perfect2}, $\Sind\ct^{\op}=\Mod\ct$. And we already know that $\Sind\ct^{\op}$
is a concretely equivalent to  variety of ordered algebras. 
\end{proof}
  
\section{Varieties as Concrete Categories}\label{sec4}%

In Section~4 we have characterized varieties $\cv$  of ordered algebras as abstract categories enriched  over $\Pos$. 
In the present section we derive a characterization of the forgetful functors $U\colon \cv \to \Pos$, i.e., varieties as concrete categories. As mentioned in Related Work, Bloom and Wright presented a characterization in \cite{BW}.
 In this section we try and compare this with our results. For categories which are exact (in the enriched sense  over $\Pos$) another such characterization is due to Kurz and Velebil (\cite{KV}, Theorem 5.7), however, we are not working with exactness in our paper.

There is not much difference between characterizing varieties abstractly or concretely:

\begin{remark}\label{re4.1} Let $\ck$ be a category with tensors.

(1) Every  generator $G$  defines a faithful functor
$$
U= \ck(G,-) \colon \ck \to \Pos
$$
with a left adjoint
$$
\phi = -\otimes G\colon \Pos \to \ck\,.
$$

(2) Every faithful functor $U\colon \ck \to \Pos$ with a left adjoint $F$ has the above form for the generator  $G=\phi 1$.

\end{remark}

Thus Theorem~\ref{T:main} has the following

\begin{corollary}\label{co4.2}
A concrete category $(\ck, U)$ over $\Pos$  is equivalent to a variety iff

\begin{enumerate}
\item[{\rm(1)}] $\ck$ has tensors,  subkernel pairs and reflexive coinserters, and

\item[{\rm(2)}] $U$ is a finitary right adjoint which reflects isomorphisms and preserves reflexive coinserters.
\end{enumerate}
\end{corollary}

\begin{proof}
(a) Let $U$ be the forgetful functor of a variety $\ck$. From Example~\ref{E:var} we know  that $U$ is a right adjoint preserving reflexive  coinserters. Following Proposition \ref{P:surj}, $U$ reflects subregular epimorphisms.
Hence it reflects isomorphisms.

\vskip 1mm
(b) Conversely, let (1) and (2) hold. Denote by $\phi \colon \Pos \to \ck$ the left adjoint of $U$ and by $T =U\phi$ the corresponding monad.  

The object $G= \phi 1$ is a strong generator: 
the functor 
$$
E\colon \ck\to \Pos^{\ck(G,G)^{\op}}
$$ 
of Notation \ref{N:str}(2) reflects 
isomorphisms because $U$ does, and we have $U\cong V.E$ for the forgetful functor $V\colon \Pos^{\ck(G,G)^{\op}}\to \Pos$. Indeed, the posets $\ck(G,K)$ are isomorphic to $UK$ (naturally in $K\in \ck$).

Since $U$ is finitary, $G$ is finitely presentable, thus abstractly finite (Example \ref{E:abstract}(3)),  and since $U$ preserves reflexive coinserters, $G$ is an effective projective.

In the proof of Theorem~\ref{T:main} we have presented a variety $\bar \ck$ and an equivalence $E\colon \ck \to \bar \ck$ assigning to every object $K$ an algebra on the poset $\ck (G,K)= UK$ and to every morphism $f\colon K\to L$ a homomorphism carried by $\ck(G,f) = Uf$. Thus, $E$ is an equivalence of concrete categories.
\end{proof}

\begin{remark}\label{re4.3}
The above corollary is related to the characterization of varieties of ordered algebras due to Bloom and Wright \cite{BW}. Their main theorem works with possibly infinitary signatures, but we  present now the formulation just for the  finitary ones to make the comparison clearer. 
\end{remark}

\begin{theorem}[Bloom and Wright \cite{BW} ]
A concrete category $(\ck, U)$ over $\Pos$ is isomorphic to a variety iff
\begin{enumerate}
\item[{\rm(1)}] $\ck$ has coinserters and 
\item[{\rm(2)}] $U$ is a finitary right adjoint which
\begin{enumerate}
\item[a.] preserves and reflects subregular epimorphisms,
\item[b.] reflects subkernel pairs, and
\item[c.] creates isomorphisms.
\end{enumerate}
\end{enumerate}
\end{theorem}
Condition (2c) is clearly related to the fact that the theorem deals with isomorphic categories rather than equivalent ones.

There does not seem to be a direct proof of one  of the above results from the other one. 


\end{document}